\documentclass[12pt]{amsart}

\usepackage[a4paper]{geometry}
\usepackage{amsmath,amsfonts,latexsym}
\usepackage{amsthm}

\usepackage{url}
\usepackage[hidelinks]{hyperref}
\usepackage{enumerate}
\usepackage{graphicx}
\usepackage{booktabs}

\numberwithin{equation}{section}

\newtheorem{theorem}{Theorem}[section]
\newtheorem{corollary}[theorem]{Corollary}
\newtheorem{lemma}[theorem]{Lemma}

\theoremstyle{definition}
\newtheorem{example}[theorem]{Example}
\newtheorem{definition}[theorem]{Definition}

\newcommand{\lc}[1]{\underline{#1}}
\newcommand{\uc}[1]{\overline{#1}}
\DeclareMathOperator{\Rack}{Rack}
\DeclareMathOperator{\Frame}{Frame}
\newcommand{\Modvi}{\ensuremath{\mathrm{MV}}}
\newcommand{\Mod}{\ensuremath{\mathrm{M}}}

\begin{document}


\title{Simplifying modular lattices by removing \\ doubly irreducible elements}
\author{Jukka Kohonen}
\address{Department of Mathematics and Systems Analysis\\
  Aalto University\\
  Espoo, Finland\\
  {\tt jukka.kohonen@iki.fi}}


\maketitle

\begin{abstract}Lattices are simplified by removing some of their
  doubly irreducible elements, resulting in smaller lattices called
  racks.  All vertically indecomposable modular racks of $n \le 40$
  elements are listed, and the numbers of all modular lattices of
  $n \le 40$ elements are obtained by P\'olya counting.  SageMath code
  is provided that allows easy access both to the listed racks, and to
  the modular lattices that were not listed.  More than 3000-fold
  savings in storage space are demonstrated.
\end{abstract}

\section{Introduction}
\label{sec:intro}

One way to begin studies of a~combinatorial family is to list all its
members.  Such listings can be quite large.  A~full listing of
unlabeled vertically indecomposable modular lattices of $n$~elements
(here denoted by~$\Modvi_n$) for all $n \le 30$ contains more than
828~million lattices, and measures over~two gigabytes in a~highly
compressed form~\cite{kohonen2019generating,eudat}.  Extending it to
bigger~$n$ would be impractical.

Another approach to a large family is by structural theorems, such as
Herrmann's theorem \cite[Hauptsatz]{herrmann1973} that represents
every modular lattice as an $S$-glued sum of its maximal complemented
intervals (see also \cite[Theorem 304]{gratzer2011}).

For practical computation, it is often useful to combine both
approaches.  A~challenge is then to find structural theorems that
balance two competing needs: being powerful enough to offer
significant computational advantages, while remaining simple enough
for efficient implementation.

Here we consider a~structural simplification that allows us to
represent the vast majority of $\Modvi_n$ as derivatives of a~smaller
set of lattices that we call \emph{racks}.  They are similar to
Gr\"atzer and Quackenbush's \emph{frames} of planar modular
lattices~\cite{gratzer2010}, but without the restriction to planarity.
For a~motivating example, consider the nonplanar modular lattices in
Figure~\ref{fig:first}.  The lattice on the left can be derived from
the one on the right by adding six doubly irreducible elements.  By
varying their placement we can obtain $64$ nonisomorphic modular
lattices, and by varying the number of added elements, we obtain many
more.

\begin{figure}[t]
  \begin{center}
    \includegraphics[width=0.35\textwidth]{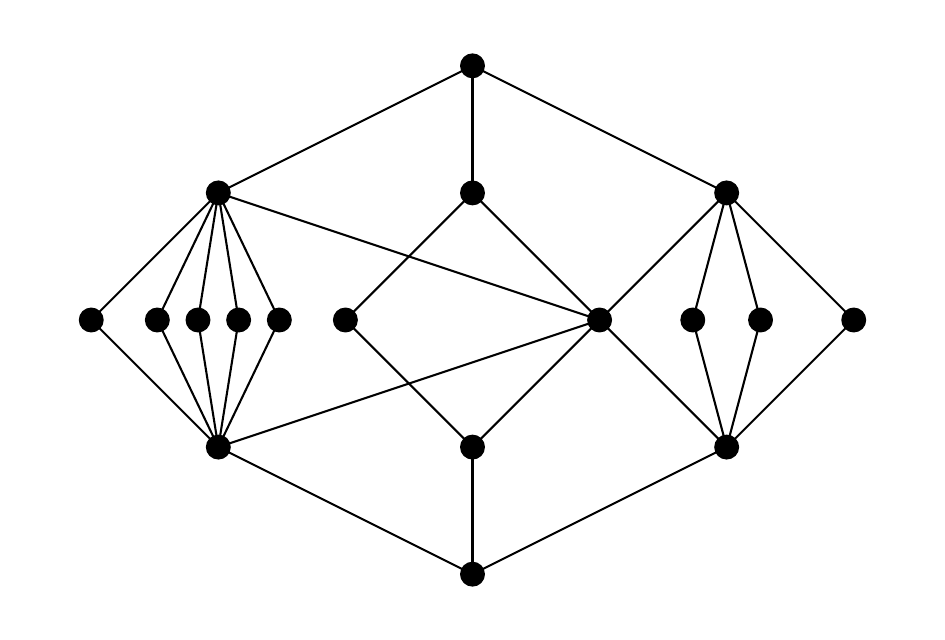}
    \hspace{0.1\textwidth}
    \includegraphics[width=0.35\textwidth]{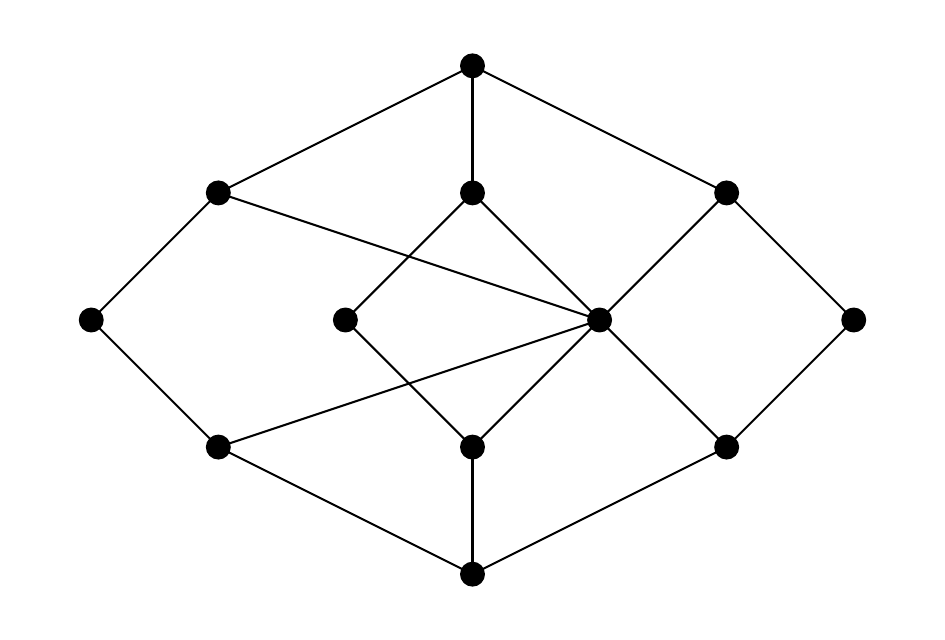}
  \end{center}
  \caption{A modular lattice and its rack.}
  \label{fig:first}
\end{figure}

In our approach only the racks are generated and stored.  Other
lattices in $\Modvi_n$ can then be P\'olya-counted, or generated
at~will by~adding doubly irreducible elements.  To~demonstrate the
viability of this approach, a~listing of all vertically indecomposable
modular racks for $n \le 40$ has been published~\cite{fairdata}.
It~contains about $1.5$ billion lattices and takes $5.7$~gigabytes to
store.  In~contrast, an~explicit listing of $\Modvi_n$ for $n \le 40$
would contain $5.2$ \emph{trillion} lattices and take about
$20$~terabytes.  Our storage savings are thus more than $3000$-fold.
It~would also take something like $60$ cpu-core-years to produce the
explicit listing.

Supplementary SageMath code~\cite{sagecode} provides easy access to
all modular lattices of $n \le 40$ elements (both vertically
decomposable and indecomposable).  The~access is through a virtual
listing whose members can be accessed sequentially, by ordinal index,
or uniformly at random.  The accessed lattices are created on demand.

Another motivation for the work is that smaller collections are more
meaningful for humans to study, and the removal of doubly irreducible
elements may help in concentrating on other structural properties of
the lattices.

\section{Definitions and basic results}
\label{sec:ears}

All our lattices are finite and nonempty.  We write $\prec$ for the
cover relation, $\uc{x}$ and $\lc{x}$ for the sets of upper and lower
covers of~$x$, and $\lvert S \rvert$ for set cardinality.  A~lattice
element~$x$ is \emph{doubly irreducible} if
$\lvert\uc{x}\rvert = \lvert\lc{x}\rvert = 1$.  An~\emph{unlabeled
  lattice} is an isomorphism class of lattices.

A lattice is \emph{vertically decomposable} if it contains
a~\emph{knot}, that is, an element distinct from top and bottom, and
comparable with every element.  Otherwise it is \emph{vertically
  indecomposable}, or briefly \emph{vi}.  Throughout this work we
focus in vi-lattices, since composing them into vertically
decomposable lattices is straightforward.

Unrestricted addition or removal of doubly irreducible elements could
severely affect the structure of a~lattice.  For example, adding such
an element between the~bottom and a~coatom in a~Boolean lattice $B_3$
yields a~nonmodular lattice.  To~keep things in check, we consider
addition and removal only at specific locations, called
\emph{decoration sites}.

For concreteness the following definitions are stated in terms of
labeled lattices, though our main interest is in unlabeled lattices.
We assume that the elements are equipped with an intrinsic linear
order (for example, they could be integers).  However, the only role
of this intrinsic order is that from a number of similarly-placed
doubly irreducible elements, we can choose a specific subset of a
given size: the ones that have the highest labels.  This ensures a
well-defined operation for computation, though for unlabeled lattices
the specific choice of elements is inconsequential.

\begin{definition}
  A \emph{decoration site} is a~pair $(a,b)$ of lattice elements with
  $\uc{a} = \lc{b}$ and $\lvert\uc{a}\rvert \ge 2$.  We say that $a$
  is its \emph{lower corner} and $b$ is its \emph{upper corner}.
  \label{def:site}
\end{definition}

\begin{definition}
  The \emph{trinkets} of a decoration site $(a,b)$ are the
  $\min(d, \; \lvert\uc{a}\rvert\!-\!2)$ highest-labeled doubly
  irreducible elements between $a$ and $b$, where $d$ is the number of
  doubly irreducible elements there.  A~decoration site without
  trinkets is \emph{empty}.
  \label{def:trinket}
\end{definition}

\begin{definition}
  A \emph{rack} is a lattice that contains no trinkets.  If~$L$ is a
  lattice, $\Rack L$ is the sublattice obtained by removing its
  trinkets.
  \label{def:rack}
\end{definition}

Put another way, $\Rack L$ is obtained by removing from each
decoration site as~many double irreducible elements as possible, while
leaving at least two (possibly doubly irreducible) elements between
the corners.  The rationale for leaving two elements is that we want
to retain the overall structure of the lattice.

\begin{example}
  Let $M_k$ denote a modular lattice that has $k$~atoms and
  length~$2$.  Then $M_k$ has $k-2$ trinkets, and
  $\Rack M_k \cong M_2$.
\end{example}

\begin{example}
  Both lattices in Figure~\ref{fig:first} contain five decoration
  sites.  On~the left, one site has $4$ trinkets, one has $2$
  trinkets, and three are empty.  On~the right, all sites are empty.
\end{example}

It is easy to see that racks of isomorphic lattices are isomorphic.
Thus we can define the rack of an unlabeled lattice as
$\Rack [L] = [\Rack L]$, where $[\;\cdot\;]$ denotes ``the isomorphism
class of''.  The lower corner of a~decoration site can be the upper
corner of another site, but otherwise decoration sites are disjoint:
no two sites can have the same lower corner, or the same upper corner;
and trinkets of one site cannot be corners or trinkets of another
site.

Now it should be emphasized that the trinket-adding operation
considered here is not novel as such.  Indeed it is a special case of
\emph{one-point extension}~\cite[\S 1.1]{gratzer2011}.  Also,
Gr\"atzer and Quackenbush~\cite{gratzer2010} define a~similar
operation when $L$ is a~planar modular lattice with a~given planar
diagram: From each interval $[a,b]$ isomorphic to some $M_k$, consider
the $k$ elements between $a$ and $b$ in the order that they appear in
the planar diagram.  Keep the first and the last, and remove the other
$k-2$ \emph{internal} elements, which are by construction doubly
irreducible.  The result is a~planar distributive lattice called
$\Frame L$.  In the planar case we have $\Frame L \cong \Rack L$, but
our construction is more general since $L$~need not be planar, and
$\Rack L$ need not be distributive.  We prefer the name \emph{rack}
because ``frame'' and ``modular frame'' are overloaded with other
meanings.

Next we observe some structural properties that are preserved upon
trinket addition and removal.

\begin{lemma}
  $L$ and $\Rack L$ have the same decoration sites, and $\Rack L$
  is a rack.
  \label{lemma:racksites}
\end{lemma}

\begin{proof}
  Let $R = \Rack L$, and let $T = L - R$.  First we show that
  all decoration sites of~$L$ are still present in~$R$.  Let $(a,b)$
  be a~decoration site in~$L$.  Since
  $\lvert\uc{a}\rvert = \lvert\lc{b}\rvert \ge 2$, the elements $a$
  and $b$ are not trinkets, so they are also present in $R$.  Also in
  $R$ the upper covers of $a$ are the same as the lower covers of $b$,
  because
  \[
    \uc{a}_R = \uc{a}_L - T = \lc{b}_L - T = \lc{b}_R.
  \]
  Here subscripts indicate the ambient lattice, so $\uc{a}_R$ means
  the upper covers of~$a$ in~$R$.  By~construction we also have
  $\lvert\uc{a}_R\rvert \ge 2$.  Thus $(a,b)$ is a~decoration site
  in~$R$.

  Next we prove that removing all trinkets from $L$ does not create
  any new decoration sites.  If some elements $(a,b)$ of $L$ are not
  a~decoration site, it is either because $\uc{a}_L \ne \lc{b}_L$, or
  because $\lvert\uc{a}_L\rvert < 2$.  In either case, removing some
  trinkets does not make $(a,b)$ a decoration site in~$R$.

  Because $\Rack L$ has the same decoration sites as~$L$, and their
  trinkets have been removed, it follows that $\Rack L$ has no
  trinkets, and is indeed a~rack.
\end{proof}

From Lemma~\ref{lemma:racksites} it follows that
$\Rack(\Rack L) = \Rack L$, or in other words, $\Rack$ is
an~idempotent operation.  This is convenient for our computations:
every modular lattice can be reduced into a~rack in a~single step of
removing all trinkets.  In the opposite direction, any modular lattice
can be created from its rack by \emph{decorating} with some trinkets.
The result of such decoration is, up to the naming of the trinkets,
uniquely determined by two things: the rack itself, and the numbers of
trinkets added to each decoration site.

\begin{lemma}
  $\Rack L$ is vertically decomposable if and only if $L$ is
  vertically decomposable.
  \label{lemma:virule}
\end{lemma}

\begin{proof}
  For the ``if'' direction, let $L$ contain a knot~$x$.  If
  $a \prec x \prec b$ in~$L$, then $\lvert\uc{a}\rvert = 1$, thus $x$
  is not a~trinket.  Then $x$ is not removed, and it is a knot in
  $\Rack L$ as~well.

  For the ``only if'' direction, let $\Rack L$ contain a knot~$x$.  If
  $u \prec x \prec v$ in~$\Rack L$, then $(u,v)$ is not a~decoration
  site because $\lvert\uc{u}\rvert = 1$.  Thus in $L$ there are no
  other elements between $u$ and $v$ than $x$~itself.  To~see that
  $x$~is a~knot in $L$ as~well, we observe that every $y \in L$ is
  either also in $\Rack L$, thus comparable to $x$; or a trinket of
  a~decoration site $(a,b)$, with either $y \prec b \preceq x$ or
  $x \preceq a \prec y$.  In either case, $y$ is comparable to $x$
  in~$L$.  It follows that $x$ is comparable to all elements of~$L$,
  and $L$~is vertically decomposable.
\end{proof}

\begin{lemma}
  If the dual of $L$ is $L^\delta$, then the dual of $\Rack L$
  is $\Rack (L^\delta)$.
  \label{lemma:dual}
\end{lemma}

\begin{proof}
  The decoration sites of $L^\delta$ are exactly the pairs $(b,a)$
  such that $(a,b)$ is a~decoration site of $L$, so $L$ and $L^\delta$
  have the same trinkets.  Thus $\Rack L$ and $\Rack (L^\delta)$ have
  the same elements.  The duality of their order is clear.
\end{proof}

\begin{lemma}
  $\Rack L$ is semimodular if and only if $L$ is semimodular.
  \label{lemma:semimod}
\end{lemma}

\begin{proof}
  We use Birkhoff's condition~\cite[Theorem 375]{gratzer2011}.  Let
  $R=\Rack L$.

  For the ``if'' direction, let $L$ be semimodular, and let
  $x,y,a \in R$ be distinct elements such that $x,y \succ a$.  Then
  also in $L$ we have $x,y \succ a$, so there exists $b \in L$ such
  that $b \succ x,y$.  Clearly $b$ is not a~trinket, so $b \in R$ and
  $b \succ x,y$ in~$R$.  Thus $R$ is semimodular.
  
  For the ``only if'' direction, let $R$ be semimodular, and let
  $x,y,a$ be distinct elements of $L$ such that $x,y \succ a$.  There
  are two cases: (1)~If $x,y \in R$, then by semimodularity there
  exists $b \in R$ such that $b \succ x,y$.  Then $b \succ x,y$ also
  in~$L$. (2)~If~$x$~or $y$ is a~trinket in $L$, then it belongs to a
  decoration site whose lower corner is~$a$.  Let $b$ be the upper
  corner of that site.  Since $\uc{a}=\lc{b}$, it follows that
  $b \succ x,y$.  In both cases Birkhoff's condition is satisfied, so
  $L$ is semimodular.
\end{proof}
  
\begin{theorem}
  $\Rack L$ is modular if and only if $L$ is modular.
  \label{thm:mod}
\end{theorem}

\begin{proof}
  Apply Lemmas~\ref{lemma:dual} and~\ref{lemma:semimod} to the duals
  of $\Rack L$ and $L$.
\end{proof}

\begin{theorem}
  Every distributive lattice is a rack.
  \label{thm:dist}
\end{theorem}

\begin{proof}
  Suppose that $L$ is a~lattice that is not a~rack.  Then it contains
  a decoration site $(a,b)$ with a~trinket $t$, and $\uc{a}=\lc{b}$
  contains at least three elements $t,u,v$.  Thus the elements
  $a,t,u,v,b$ are a~diamond (a~sublattice isomorphic to $M_3$).
  By~\cite[Theorem~102]{gratzer2011}, $L$~is not distributive.
\end{proof}

\begin{theorem}[Gr\"atzer and Quackenbush~\cite{gratzer2010}]
  If~$L$ is a planar modular lattice, then $\Rack L$ is planar and
  distributive.
  \label{thm:planar}
\end{theorem}

\begin{corollary}
  Every planar modular rack is distributive.
  \label{cor:planar}
\end{corollary}

In contrast to Theorem~\ref{thm:dist} and Corollary~\ref{cor:planar},
a~modular rack need not be distributive, and a~distributive rack need
not be planar.  Some examples can be found in Figure~\ref{fig:smalls},
where all vertically indecomposable modular racks of $1$~to~$13$
elements are displayed.  Nondistributive racks include 10.0 (no
decoration sites) and 12.22 ($B_3^+$ in~\cite{gratzer2010}; has five
decoration sites).  Distributive nonplanar racks include 8.0 ($B_3$,
no decoration sites) and 10.1 (has one decoration site).

\begin{figure}[p]
  \begin{center}
    \includegraphics[height=0.9\textheight]{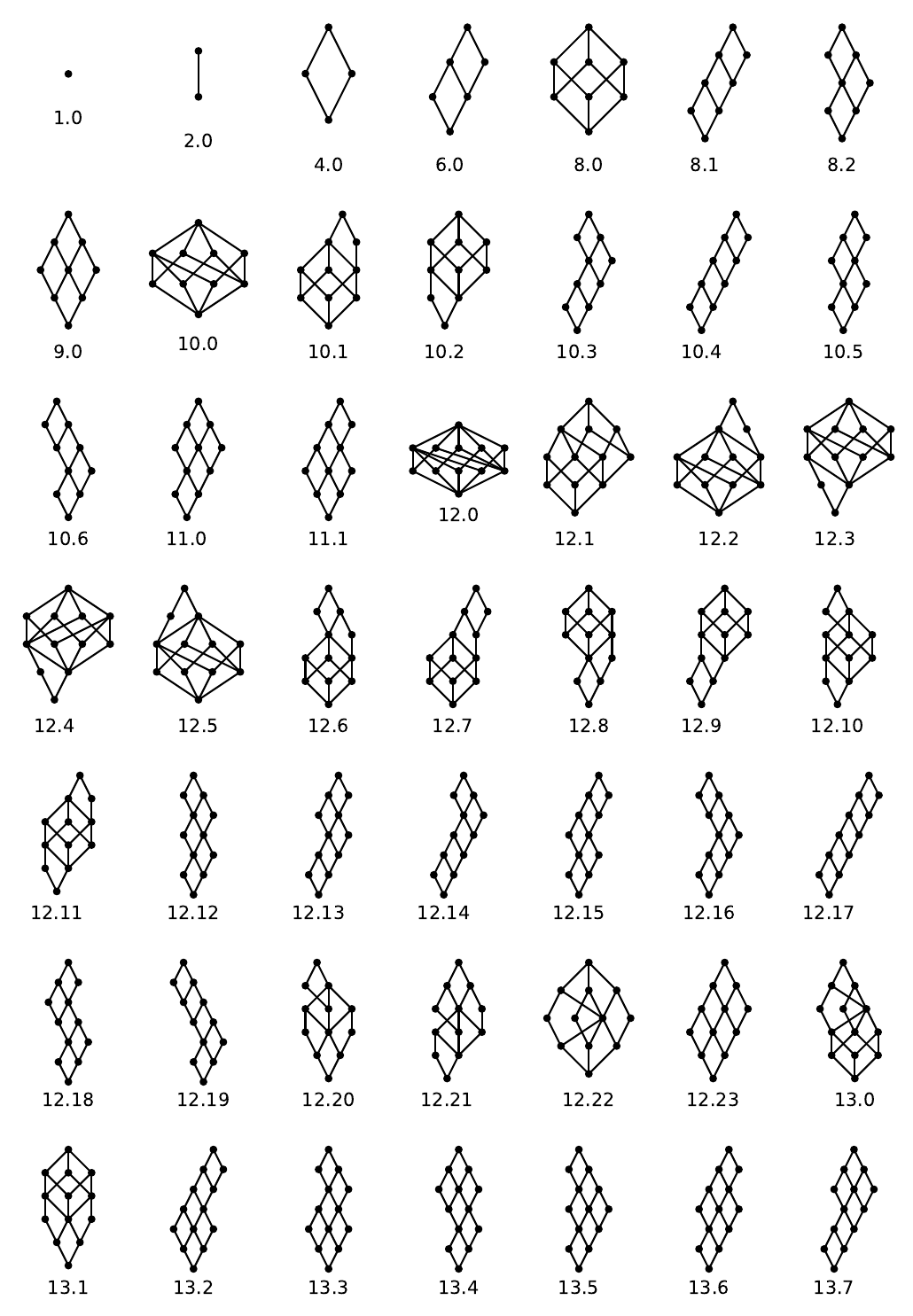}
  \end{center}
  \caption{All modular vi-racks of $1$ to $13$ elements, labeled with
    the number of elements and an ordinal index.}
  \label{fig:smalls}
\end{figure}

Planar distributive lattices are well understood.  Every planar
distributive lattice can be obtained from the direct product of two
finite chains by removing two arbitrarily shaped ``corners'' from left
and right~\cite{gratzer2007}.  With this characterization they are
easily counted: OEIS gives a~simple recurrence and counts them up~to
$n=1000$~\cite[A343161]{oeis}.  In the nonplanar, nondistributive case
we~have no~such nice characterization, and computational methods are
needed to obtain all racks.

\section{Decoration under symmetry}
\label{sec:decoration}

\begin{figure}[b]
  \begin{center}
    \raisebox{0.5cm}{\includegraphics[width=0.30\textwidth]{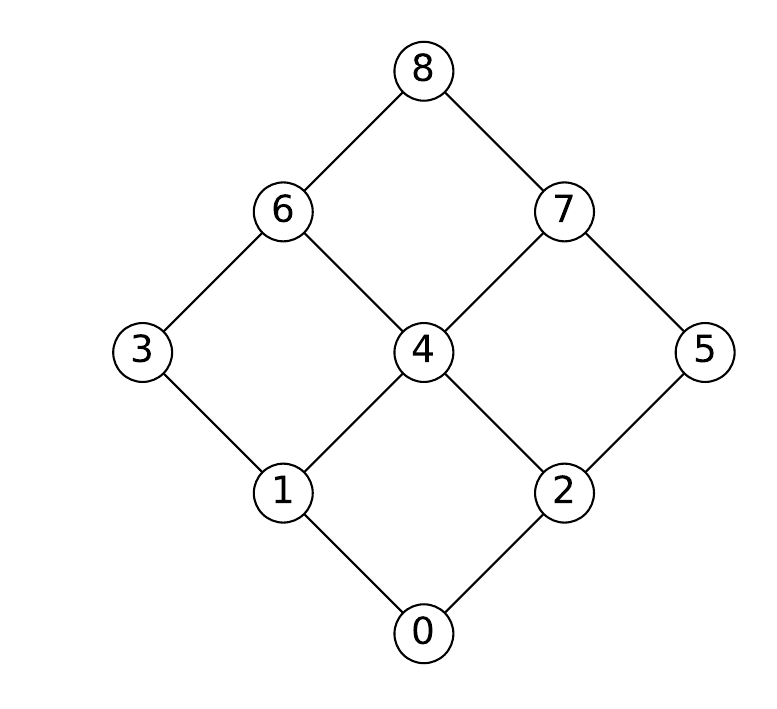}}
    \hspace{0.05\textwidth}
    \includegraphics[width=0.48\textwidth]{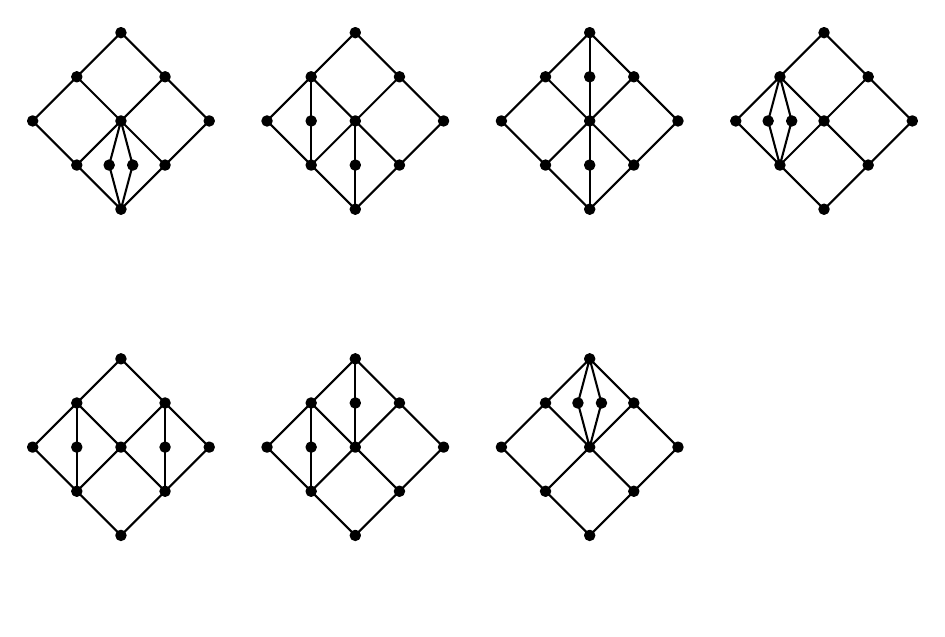}
  \end{center}
  \caption{A modular rack and its seven nonisomorphic decorations with
    two trinkets.}
  \label{fig:mirror}
\end{figure}

From the previous section we know that every modular lattice $L$ can
be created by \emph{decorating} a modular rack with trinkets.  Let us
now investigate two questions: Given a~rack~$R$ and an integer~$m$,
how many nonisomorphic modular lattices can be created by placing
$m$~trinkets, and how can we actually construct them?

These tasks are complicated by the symmetries of the rack.  Consider
the rack in Figure~\ref{fig:mirror}.  It has four decoration sites:
$(0,4)$, $(1,6)$, $(2,7)$ and $(4,8)$.  We can treat them as boxes
where indistinguishable balls, or trinkets, are distributed.  With
(say) two trinkets, the number of ways is $\binom{2+4-1}{2} = 10$ by
the stars-and-bars method, but only $7$ of the resulting lattices are
nonisomorphic.

In order to count the nonisomorphic results, we employ P\'olya
counting as follows.  Suppose that we are decorating a~rack that has
$k$ decoration sites, and their symmetry group is~$G$.  Let $Z$ be the
cycle index of~$G$, that is, the polynomial
\[
  Z(t_1,\dots,t_k) = \frac{1}{\lvert G \rvert} \sum_{g \in G}
  t_1^{c_1(g)} t_2^{c_2(g)} \cdots t_k^{c_k(g)},
\]
where $c_i(g)$ is the number of cycles of length~$i$ in
permutation~$g$.  Define the figure-counting series
\[
  A(x) = 1+x+x^2+x^3+\ldots = 1/(1-x),
\]
indicating that each decoration site can be allocated any nonnegative
integer number of trinkets.  Now by the Cycle Index
Theorem~\cite[p.~77]{cameron2015}, the series
\[
  B(x) = Z(A(x), A(x^2), \dots, A(x^k))
\]
is the so-called function-counting series: the coefficient of its
$x^m$ term is the number of nonisomorphic ways to distribute a~total
of $m$ trinkets to the $k$ decoration sites.

\begin{example}
  The rack in Figure~\ref{fig:mirror} has mirror symmetry.  Let us
  refer to the decoration sites by their lower corners.  $G$~has two
  elements, the identity $(0)(1)(2)(4)$ and the mirroring
  $(0)(1 \; 2)(4)$, so $Z = \frac{1}{2}t_1^4 + \frac{1}{2}t_1^2t_2$.
  We obtain
  \[
    B(x) = 1 + 3x + 7x^2 + 13x^3 + 22x^4 + 34x^5 + \dots,
  \]
  which tells us that there is $1$ decoration with zero trinkets, $3$
  nonisomorphic decorations with one trinket, $7$ with two trinkets,
  and so on.
\end{example}

\begin{example}
  The rack in Figure~\ref{fig:complicated} is more complicated.  It
  has eleven decoration sites, and their symmetry group is isomorphic
  to the dihedral group $D_4$.  The~group fixes three sites and moves
  eight sites in a~nontrivial way.  It~would be tedious to work out
  the symmetry and count the decorations manually.  But in the
  supplementary SageMath code we have the function
  \verb@count_decorations@ that implements the method described above.
  With this function we find (in a~few milliseconds) that if we were
  to decorate this rack with, say, $20$~trinkets, we would obtain
  exactly $5\;371\;900$ nonisomorphic modular lattices.
\end{example}

\begin{figure}[b]
  \begin{center}
    \includegraphics[width=0.44\textwidth]{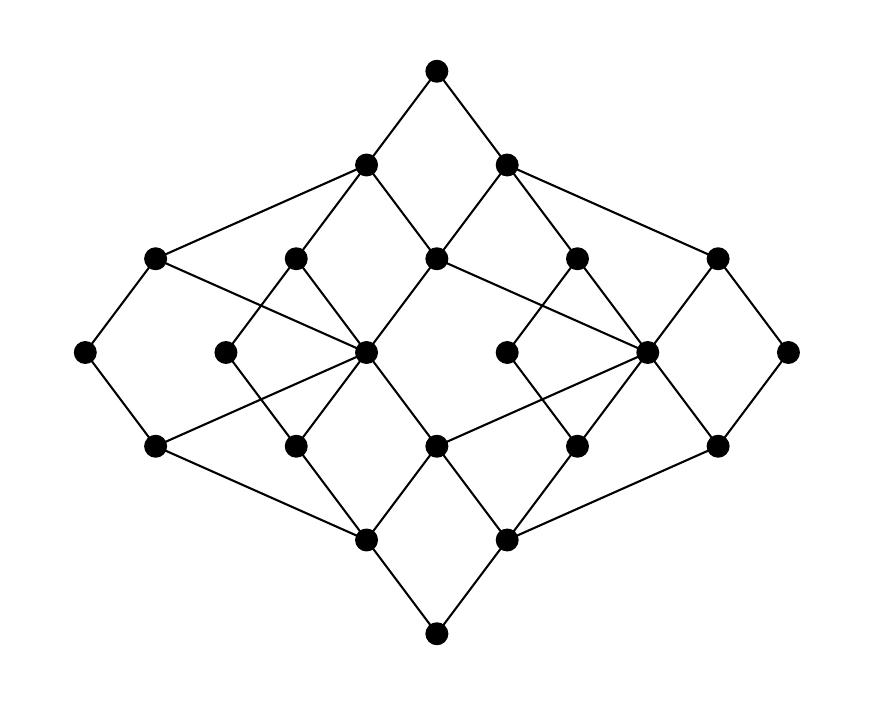}
  \end{center}
  \caption{A rack with eleven decoration sites.}
  \label{fig:complicated}
\end{figure}

If we require the actual lattices (and not just their count), then we
need a~different tool.  For this we use
\verb@IntegerVectorsModPermutationGroup@ developed by
Borie~\cite{borie2013} and incorporated into the SageMath
Combinatorics library.  Given the symmetry group of a~rack's
decoration sites, and a~number of trinkets $m$, this tool lists the
different ways of distributing $m$ balls to boxes under that symmetry.
It is then straightforward to create the lattices by adding those
numbers of trinkets to the sites.  This is implemented in our function
\verb@list_decorations@.  Explicit listing is of course much slower
than counting.

\begin{example}
  With the rack of Figure~\ref{fig:mirror} (left), the possible
  allocations of two trinkets to the four decoration sites, subject to
  the symmetry, are $(2, 0, 0, 0)$, $(1, 1, 0, 0)$, $(1, 0, 0, 1)$,
  $(0, 2, 0, 0)$, $(0, 1, 1, 0)$, $(0, 1, 0, 1)$, and $(0, 0, 0, 2)$.
  By adding these numbers of trinkets, we obtain the modular lattices
  in Figure~\ref{fig:mirror} (right).
\end{example}

\section{Generating and classifying the racks}
\label{sec:computation}

Computations were performed in phases.  First, all unlabeled modular
vi-racks of $n \le 40$ elements were listed using essentially the same
C++ program that has been used in earlier
works~\cite{kohonen2019generating,kohonen2022}.  Conditions were added
to the program so that lattices with trinkets are not generated.  This
phase took $70.4$ cpu-core days in total, running on a~variety of AMD
EPYC 7713 and Intel Xeon processors with nominal clock rates mostly
between $2.0$ and $2.5$~GHz.  The parts for $n=36,37,38,39,40$ took
$1.6$, $3.6$, $8.1$, $17.7$ and $38.5$ cpu-core days, respectively,
showing roughly $2.2$-fold growth when $n$ increases by one.

The speed benefits of our method were tested as follows.  All
unlabeled modular vi-lattices of $30$~elements were explicitly listed
using the same C++ program.  This took $27.7$ hours; in comparison,
listing only the 30-element modular vi-racks took $0.36$~hours.  The
time savings grow with~$n$, and we estimate, very roughly, that
an~explicit listing for $n=40$ would have taken $60$ cpu-core years.

In the second phase the modular racks were analyzed and postprocessed.
For each modular rack, the decoration sites were located, and the
cycle index of their symmetry group was computed.  A~tally of racks
was kept for each different cycle index encountered.  The racks were
also converted to a~canonical form for ease of later use, and stored
in XZ-compressed dig6 format~\cite{xz}.  This phase took 237.6
cpu-core days, somewhat more than the first phase, but this can be
accounted to the relatively slow SageMath code that was used.
A~faster (e.g.\ C++) program for this phase could be written if
necessary.

In the third phase all unlabeled modular vi-lattices of $n \le 40$
elements were counted.  This is a~matter of seconds, because all that
remains to do in this phase is to combine the counts of racks, from
each cycle index, with the numbers of decorations per rack.  More
precisely, we have
\[
  \lvert \Modvi_n \rvert = \sum_{k=1}^n \; \sum_{Z \in {\mathcal Z}(k)}
  R(k,Z) \cdot D(Z, n-k),
\]
where ${\mathcal Z}(k)$ is the set of all different cycle indices in
$k$-element modular vi-racks, $R(k,Z)$ is the number of unlabeled
$k$-element modular vi-racks whose decoration sites have a~symmetry
group with cycle index $Z$, and $D(Z, n-k)$ is the number of
decorations of each such rack with $n-k$ trinkets.  Here $R(k,Z)$
comes from our tallying in the second phase, and $D(Z,n-k)$ is
calculated with the method described in Section~\ref{sec:decoration}.
The summation is fast, because it does not involve very many terms:
even $\lvert {\mathcal Z}(40) \rvert$ is only $1614$ (see
Table~\ref{tbl:counts}).  Although there are hundreds of millions of
different racks, they can be grouped into a~relatively few types by
their decoration symmetry.

Finally the numbers of unlabeled modular lattices of $n$ elements
($\Mod_n$) were calculated using the well-known
recurrence~\cite{jipsen2015}
\[
  \lvert \Mod_n \rvert =
  \sum_{j=2}^{n} \lvert \Modvi_j \rvert \cdot \lvert \Mod_{n-j+1} \rvert,
\]
which counts the ways of composing the vi-lattices vertically.


Several consistency checks were performed in order to increase the
reliability of the computational results.  First, the racks were
generated twice on different computer systems.  The output files were
verified to be byte-by-byte identical by comparing their MD5
checksums.

Secondly we counted the occurrences of each rank sequence in the rack
listings, and verified that all those counts are consistent with
duality.  For example, among all unlabeled modular vi-racks of $40$
elements, exactly $4\;265$ have the rank sequence $1,3,5,6,7,7,6,4,1$,
and another $4\;265$ (their duals) have the reverse of that, namely
$1,4,6,7,7,6,5,3,1$.

Thirdly the numbers of modular vi-lattices were verified against
previous results, which went to $n=30$~\cite{kohonen2019generating}
and to $n=35$~\cite{kohonen2022}.  We must note that those two
previous countings are based on the same underlying lattice-listing
C++ program that was also used here.  However, the combinatorial
methods are quite different, so we are reaching the same numbers by
three different methods.

Fourthly, a~listing of $n$-element modular vi-racks must contain all
distributive vi-lattices of that size.  We~scanned the rack lists for
distributive lattices of up to $40$ elements, and verified that the
counts match earlier results~\cite{erne2002,kohonen2022}.

Finally, the rack listings and their decorations were comprehensively
compared, lattice by lattice, against the explicit listings of modular
vi-lattices published earlier~\cite{kohonen2019generating,eudat}.
This test went both ways.  We scanned the explicit listings for racks,
and verified that they are the same racks as those listed in the
present work.  Also, we listed all decorations from our racks, and
verified that this exactly re-creates the modular vi-lattices that
were listed earlier, up to isomorphism.  This test was computationally
intensive, and was performed only up to 19 elements.

\section{Numerical results}
\label{sec:numerical}

The counting results are displayed in Table~\ref{tbl:counts}.  The
second column contains $\lvert {\mathcal Z}(n) \rvert$, the number of
different cycle indices of the decoration symmetries in $n$-element
modular vi-racks.  The last three columns contain the numbers of
unlabeled modular vi-racks, modular vi-lattices, and modular lattices,
respectively.  The last two columns were previously known up to $n=35$
from a~different method of counting~\cite{kohonen2022}.

Figure~\ref{fig:counts} illustrates how the numbers of unlabeled
lattices in different families depend on the number of elements.
The~data are from Table~\ref{tbl:counts} and the OEIS entries A072361
and A345734~\cite{oeis}.  Although precise asymptotics are not known,
empirically the growth rate of modular vi-racks is about
$\Theta(1.9^n)$, and it is closer to distributive vi-lattices than to
modular vi-lattices.  We~can see that much of the apparent multitude
of modular vi-lattices is just decoration.

\begin{table}[p]
  \caption{Numbers of unlabeled modular lattices of $n$~elements.}
  \begin{center}
  \begin{tabular}{rrrrr}
    \toprule
    $n$ & Cycle indices & Mod. vi-racks & Mod. vi-lattices & Mod. lattices \\
        &               &               & A342132          & A006981~\cite{oeis} \\
    \midrule
   1 &    1 &            1 &               1 &                  1 \\
   2 &    1 &            1 &               1 &                  1 \\
   3 &    0 &            0 &               0 &                  1 \\
   4 &    1 &            1 &               1 &                  2 \\
   5 &    0 &            0 &               1 &                  4 \\
   6 &    1 &            1 &               2 &                  8 \\
   7 &    0 &            0 &               3 &                 16 \\
   8 &    2 &            3 &               7 &                 34 \\
   9 &    1 &            1 &              12 &                 72 \\
  10 &    3 &            7 &              28 &                157 \\
  11 &    1 &            2 &              54 &                343 \\
  12 &    7 &           24 &             127 &                766 \\
  13 &    2 &            8 &             266 &               1718 \\
  14 &    8 &           70 &             614 &               3899 \\
  15 &   13 &           44 &            1356 &               8898 \\
  16 &   12 &          215 &            3134 &              20475 \\
  17 &   16 &          173 &            7091 &              47321 \\
  18 &   23 &          711 &           16482 &             110024 \\
  19 &   27 &          657 &           37929 &             256791 \\
  20 &   33 &         2367 &           88622 &             601991 \\
  21 &   42 &         2561 &          206295 &            1415768 \\
  22 &   57 &         7989 &          484445 &            3340847 \\
  23 &   60 &         9745 &         1136897 &            7904700 \\
  24 &   80 &        27540 &         2682451 &           18752943 \\
  25 &   98 &        36744 &         6333249 &           44588803 \\
  26 &  115 &        95975 &        15005945 &          106247120 \\
  27 &  140 &       137895 &        35595805 &          253644319 \\
  28 &  179 &       337911 &        84649515 &          606603025 \\
  29 &  212 &       514821 &       201560350 &         1453029516 \\
  30 &  251 &      1200282 &       480845007 &         3485707007 \\
  31 &  318 &      1915896 &      1148537092 &         8373273835 \\
  32 &  375 &      4291336 &      2747477575 &        20139498217 \\
  33 &  440 &      7113503 &      6579923491 &        48496079939 \\
  34 &  549 &     15430316 &     15777658535 &       116905715114 \\
  35 &  655 &     26356273 &     37871501929 &       282098869730 \\
  36 &  772 &     55742330 &     90998884153 &       681357605302 \\
  37 &  944 &     97509982 &    218856768070 &      1647135247659 \\
  38 & 1133 &    202116488 &    526836817969 &      3985106742170 \\
  39 & 1319 &    360362439 &   1269255959032 &      9649048527989 \\
  40 & 1614 &    735089580 &   3060315929993 &     23379906035595 \\
    \bottomrule
  \end{tabular}
  \end{center}
  \label{tbl:counts}
\end{table}

\begin{figure}
  \begin{center}
    \includegraphics[width=0.8\textwidth]{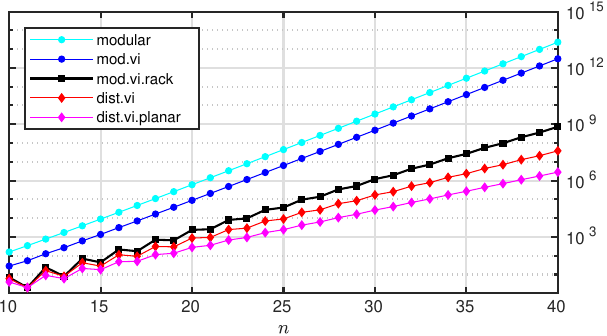}
  \end{center}
  \caption{Numbers of modular lattices, modular vi-lattices, modular
    vi-racks, distributive vi-lattices, and planar distributive
    vi-lattices of $n$ elements (all up to isomorphism).}
  \label{fig:counts}
\end{figure}

\section{SageMath implementation}
\label{sec:sage}

A library of SageMath code was developed to support the use of the
rack listings~\cite{sagecode}.  An~overview of its capabilities is
given here.

The library contains utility functions for reading lattices from text
files, adding and removing trinkets, finding the decoration sites of
a~lattice, and so on.  But~the central part of the library is the
combinatorial core of our approach: given a~rack and a~number of
trinkets, functions \verb@count_decorations@ and
\verb@list_decorations@ can count and list all the ensuing
nonisomorphic decorations.  Counting is much faster than listing,
because it is done through P\'olya counting as~described in
Section~\ref{sec:decoration}.  Instead of listing all decorations, one
can also ask for a~specific decoration by its ordinal index.  This is
useful, for example, if one wants to sample uniformly at~random from
a~large set of decorations.

\begin{figure}[t]
  \begin{center}
    \includegraphics[height=0.48\textwidth]{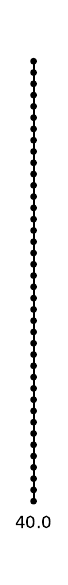}
    \hspace{0.02\textwidth}
    \includegraphics[height=0.48\textwidth]{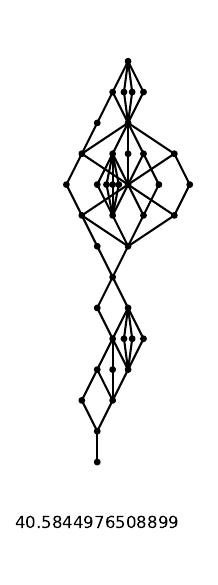}
    \hspace{0.01\textwidth}
    \includegraphics[height=0.48\textwidth]{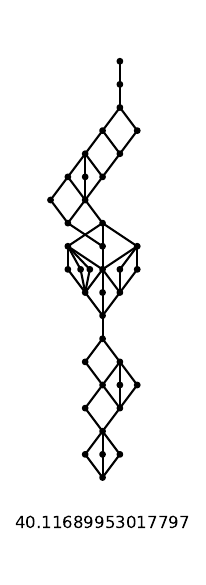}
    \hspace{0.01\textwidth}
    \includegraphics[height=0.48\textwidth]{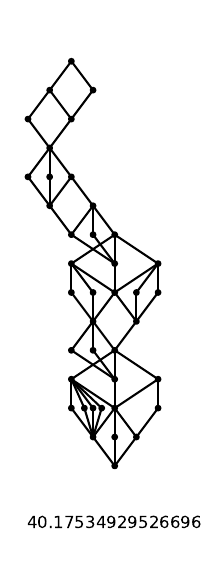}
    \hspace{0.01\textwidth}
    \includegraphics[height=0.48\textwidth]{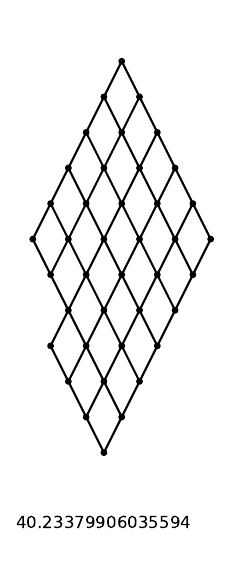}
  \end{center}
  \caption{An evenly spaced sample of five 40-element modular lattices,
    from the first to the last in our virtual listing.}
  \label{fig:sample}
\end{figure}

Both the file access and the combinatorial core are encapsulated into
classes belonging to the category \verb@FiniteEnumeratedSets@.  From
the outside, such a~class appears as a~virtual list whose members can
be \emph{iterated} (accessed sequentially) and \emph{unranked}
(accessed by an integer index).  It~is up to the implementation how
the members are produced when asked for.  We~have nested levels of
wrappers.  The lowest-level wrapper encapsulates a~rack listing that
resides in an~XZ-compressed text file.  The~unranker seeks to the
correct position and reads one lattice from there.  Further wrappers
encapsulate decoration and vertical composition.  The unrankers obtain
the appropriate racks from the file wrapper, decorate with trinkets,
and compose vertically as needed.  As~a result we have easily
accessible virtual lists of modular vi-lattices and modular lattices
of $n \le 40$ elements.

\begin{example}
  The following code asks for a nonrandom, approximately evenly spaced
  sample of five modular lattices of $40$~elements, ranging from the
  first one to the last one in the virtual listing, according to an
  intrinsic order.  The sample is promptly returned without ever
  having to construct $23$~trillion lattices.
\begin{verbatim}
    sage: M = ModularLattices(40)
    sage: card = M.cardinality(); print(card)
    23379906035595
    sage: LL = list(M[round((card-1)*i/4)] for i in [0,1,2,3,4]]
\end{verbatim}
  The first line takes five seconds on a~laptop computer, as~it opens
  the files and sets things up for fast retrieval.  The next lines
  take less than a~second.  The sample is displayed in
  Figure~\ref{fig:sample}.  Note that in our virtual listing, racks
  are ordered by the number of decoration sites, and vertically
  decomposable lattices are ordered by the size of the lowest
  vertically indecomposable component.  So~it is not a~coincidence
  that the first lattice in the listing is the chain, and the last one
  is a~planar distributive lattice with many (empty) decoration sites.
\end{example}

We note in passing that for automatically generated modular lattices,
SageMath's \verb@plot@ often produces diagrams that contain
unnecessarily many crossing edges.  Our \verb@lattice_plot@ attempts
to draw somewhat prettier lattice diagrams.  It~recognizes trinkets
and hangs them out in an~aptly trinketlike fashion.  All lattice
diagrams in this paper, including Figure~\ref{fig:sample}, were
automatically produced with this function.

\section{Related work}
\label{sec:related}

This work was initially inspired by the simple observation that
modular lattices contain lots of doubly irreducible elements.  Several
previous works share the idea of adding single elements or other
simple features to a~lattice, sometimes using enumerative
combinatorics to count the ways of doing that.

As already mentioned, Gr\"atzer and Quackenbush described a reduction
similar to ours, where internal doubly irreducible elements are
removed from a~planar modular lattice~\cite{gratzer2010}.  The present
work removes the restriction to planarity, so that we can represent
all modular lattices.

Jipsen and Lawless proved an~$\Omega(2^n)$ lower bound for the number
of unlabeled modular lattices using a recursive construction, where
each step either extends the lattice vertically, or adds a~doubly
irreducible element~\cite{jipsen2015}.

Bhavale and Waphare studied dismantlable lattices whose reducible
elements are comparable to each other; such lattices consist of
a~single main chain with attached side chains.  The possible
placements of the side chains were counted using binomial
coefficients~\cite{bhavale2020}.

The present author studied rank-three graded lattices (without
modularity) and reduced them by removing all doubly irreducible atoms;
those atoms were then treated as indistinguishable balls to be placed
into partially distinguishable boxes, leading to P\'olya
counting~\cite{kohonen2019counting}.

\section{Concluding remarks}
\label{sec:conclusion}

We have seen that the removal of some doubly irreducible elements can
lead to big savings in computation and storage.  Both theoretical and
practical tools were needed to make it happen.  Structural theorems
provide the foundation for such work, but algorithms and computations
bring the theorems to~life.  The operation considered here, where
racks are decorated with trinkets, is at the same time a special case
of one-point extension (for lattices in general), and a generalization
of Gr\"atzer and Quackenbush's operation of adding eyes (for planar
modular lattices).  The level of specialization has been chosen so as
to be amenable to efficient computation (counting and access to
individual lattices).

From Figure~\ref{fig:smalls} one may observe that many of our racks
could be composed from smaller components by gluing constructions.
For example, lattice 10.1 is the Dilworth gluing of a $B_3$ with a
$B_2$ over a two-element lattice, and lattice 9.0 could be represented
as an $S$-glued system of four copies of $B_2$ (cf. Figures~1 and~2 of
Day and Freese~\cite{day1990} for a similar gluing of copies of
$B_3$).  Conceivably, such decompositions could be used to further
simplify and reduce our exhaustive listings of lattices.  The
challenge then is to employ gluing in a computationally efficient
manner, so that one can still effectively access the individual
lattices at will.  This is an interesting topic for further study.

This work provides a~virtual \emph{listing} of modular lattices.
For~some uses one might want a~\emph{database} that would allow
efficient queries according to various criteria, such as number of
levels, numbers of elements at specific levels, whether the lattice is
slim, complemented, planar, and so on.  This could be similar to what
the online House of Graphs provides for graphs~\cite{coolsaet2023}.
An~intriguing prospect is a~virtual database, one that would represent
a~large collection in terms of a~smaller explicit collection,
similarly to what was done here, and would still support efficient
queries.  For this to work, one would have to take into account how
some properties are preserved and others affected by whatever
structural reductions one is employing.  Making it click together
might involve some interesting combinatorics.

\subsection*{Acknowledgments}

The author wishes to thank the anonymous referees for their
illuminating comments.  Thanks are also due to the numerous authors of
the computational tools that were essential for this work.
SageMath~\cite{sagemath} and its component libraries facilitate
computation with graphs, symmetry groups, cycle indices, and symbolic
power series.  The underlying rack-listing program relies on
Nauty~\cite{nauty} for isomorph avoidance.  XZ~\cite{xz} compresses
the lattice files efficiently while enabling fast random access to
individual lattices.  Computational resources were generously provided
by CSC --- IT Center for Science and by the Aalto Science-IT project.


\end{document}